\documentclass[11pt, a4paper,leqno]{amsart}
\usepackage{amsmath,amsthm,amscd,amssymb,amsfonts, amsbsy}
\usepackage{latexsym}
\usepackage{txfonts}
\usepackage{exscale}
\usepackage{mathrsfs}
\usepackage[titletoc]{appendix}

\usepackage[utf8]{inputenc}

\renewcommand\labelenumi{(\roman{enumi})}
\renewcommand\theenumi\labelenumi

\usepackage[colorlinks,citecolor=red,pagebackref,hypertexnames=false, breaklinks]{hyperref}
\usepackage[usenames,dvipsnames]{color}

\calclayout
\allowdisplaybreaks

\def \Int{\,\Int\,}
\def \iint{\int\!\!\!\int}

\def\Xint#1{\mathchoice
    {\XXint\displaystyle\textstyle{#1}}%
    {\XXint\textstyle\scriptstyle{#1}}%
    {\XXint\scriptstyle\scriptscriptstyle{#1}}%
    {\XXint\scriptscriptstyle\scriptscriptstyle{#1}}%
    \!\int}
    \def\XXint#1#2#3{{\setbox0=\hbox{$#1{#2#3}{\int}$}
    \vcenter{\hbox{$#2#3$}}\kern-.5\wd0}}
    \def\fint{\Xint-}

\def\Xint#1{\mathchoice
    {\XXint\displaystyle\textstyle{#1}}%
    {\XXint\textstyle\scriptstyle{#1}}%
    {\XXint\scriptstyle\scriptscriptstyle{#1}}%
    {\XXint\scriptscriptstyle\scriptscriptstyle{#1}}%
    \!\int}
    \def\XXint#1#2#3{{\setbox0=\hbox{$#1{#2#3}{\int}$}
    \vcenter{\hbox{$#2#3$}}\kern-.5\wd0}}

\def \C{ \mathbb{C} }

\def \R{ \mathbb{R} }

\newcommand{\ree}{\mathbb{R}^{n+1}}

\def\div{\mathop{\operatorname{div}}\nolimits}

\DeclareMathOperator*{\essinf}{ess\,inf}

\newcommand{\loc}{\mathrm{loc}}

\renewcommand{\chi}{\mathbf{1}}

\theoremstyle{plain}
\newtheorem{theorem}[equation]{Theorem}
\newtheorem{lemma}[equation]{Lemma}
\newtheorem{corollary}[equation]{Corollary}
\newtheorem{proposition}[equation]{Proposition}

\theoremstyle{definition}
\newtheorem{definition}[equation]{Definition}

\theoremstyle{remark}
\newtheorem{remark}[equation]{Remark}

\numberwithin{equation}{section}

\numberwithin{equation}{section}

\numberwithin{equation}{section}

\def \re{ \mathbb{R} }

\def\essinf{\mathop\mathrm{\,ess\,inf\,}}

\begin{document}
\allowdisplaybreaks
\author{Jos\'e Mar\'ia Martell}
\address{Jos\'e Mar\'ia Martell
\\
Instituto de Ciencias Matem\'aticas CSIC-UAM-UC3M-UCM
\\
Consejo Superior de Investigaciones Cientificas
\\
C/ Nicol\'as Cabrera, 13-15
\\
E-28049 Madrid, Spain} \email{chema.martell@icmat.es}

\author{Pierre Portal}
\address{Pierre Portal
Mathematical Sciences Institute
\\
Australian National University
\\
Ngunnawal and Ngambri Country
\\
Canberra ACT 0200, Australia}

\email{pierre.portal@anu.edu.au}

\title{A note on Rubio de Francia's extrapolation in tent spaces and applications}

	\thanks{The first author acknowledges financial support from MCIN/AEI/ 10.13039/501100011033 grants CEX2019-000904-S and PID2019-107914GB-I00. This research was primarily conducted when the second author visited the first author at the Instituto de Ciencias Matem\'aticas (ICMAT) in Madrid. The perfect working environment provided by the ICMAT is very much appreciated. 	
	}


\date{\today}

\begin{abstract}
The Rubio de Francia extrapolation theorem is a very powerful result which states that in order to show that certain operators satisfy weighted norm inequalities with Muckenhoupt weights it suffices to see that the corresponding inequalities hold for
some fixed exponent, for instance $p=2$. In this paper we extend this result and show that this extrapolation principle allows one to obtain weighted estimates in tent spaces. From our extrapolation result we automatically derive new estimates (and reprove some other) concerning Calder\'on-Zygmund operators, operators associated with the Kato conjecture, or fractional operators. 
\end{abstract}

\maketitle

\tableofcontents

\section{Introduction}

The celebrated extrapolation theorem of Rubio de Francia \cite{RdF} gives that, if an operator $T$ satisfies 
\begin{equation*}
\|Tf\|_{L^{p_0}(w)}\lesssim \|f\|_{L^{p_0}(w)}, \quad
\text{for some }p_0\in[1,\infty)\text{ and every }w\in A_{p_0}, 
\end{equation*}
then
\begin{equation*}
\|Tf\|_{L^{p}(w)}\lesssim\|f\|_{L^{p}(w)},
\quad \text{for every }p\in(1,\infty)\text{ and every }w\in A_{p}.
\end{equation*}
Here, $A_{p}$ denotes the class of Muckenhoupt weights in $\re^n$ with the underlying Lebesgue measure and $L^{p}(w)$ denotes the associated weighted Lebesgue space with underlying measre $w(x)\,dx$. See Section \ref{section:prelim} for precise definitions. 

The theory of Muckenhoupt weights has been extensively studied, and one of its most basic features is the fact that $w\in A_{p}$ if and only if the Hardy-Littlewood maximal function $M$ is bounded on $L^p(w)$ whenever $p\in(1,\infty)$; and for $p=1$, $w\in A_{1}$ if and only if $M$  maps continuously $L^1(w)$ into $L^{1,\infty}(w)$. 

A modern treatment of extrapolation can be found in \cite{CMP} where it is shown that operators play no role in extrapolation, indeed one can reformulate it in  terms of pairs of functions. To elaborate on this, let $\mathcal{F}$ be some family of pairs $(f,g)$ of non-negative measurable functions. Rubio de Francia's extrapolation theorem can  then be rewritten as follows: if 
\begin{equation}\label{eq:intro-RdF-1}
\|f\|_{L^{p_0}(w)}\lesssim \|g\|_{L^{p_0}(w)}, \quad\text{for some }p_0\in[1,\infty),\text{ every }w\in A_{p_0},
\text{ and every }(f,g)\in\mathcal{F},
\end{equation}
then
\begin{equation}\label{eq:intro-RdF-2}
\|f\|_{L^{p}(w)}\lesssim \|g\|_{L^{p}(w)}, \quad\text{for every }p\in(1,\infty),\text{ every }w\in A_{p},
\text{ and every }(f,g)\in\mathcal{F}.
\end{equation}
This formulation is very useful and has profound consequences. For example, suppose that \eqref{eq:intro-RdF-1} holds (hence so does \eqref{eq:intro-RdF-2}) and we wish to show the following family of vector-valued inequalities:
\begin{equation}\label{eq:ewreavce}
\Big\|\Big(\sum_j f_j^r\Big)^{\frac1r}\Big\|_{L^q(w)}
\lesssim
\Big\|\Big(\sum_j g_j^r\Big)^{\frac1r}\Big\|_{L^q(w)}
,\quad \{(f_j,g_j)\}_j\subset\mathcal{F},
\end{equation}
with $1<q,r<\infty$ and $w\in A_q$. Fixing $r\in (1,\infty)$, there is some choice of $q$ for which this estimate follows easily. Indeed, if $q=r$, we can interchange the sum and the integral to obtain that for any $w\in A_r$, \eqref{eq:intro-RdF-2} (applied to each $(f_j,g_j)$ with $p=r$ and $w\in A_r$) yields
\begin{equation}\label{eq:wregrsbrtg}
\Big\|\Big(\sum_j f_j^r\Big)^{\frac1r}\Big\|_{L^r(w)}^r
=
\sum_j \|f_j\|_{L^r(w)}^r
\lesssim
\sum_j \|g_j\|_{L^r(w)}^r
=
\Big\|\Big(\sum_j g_j^r\Big)^{\frac1r}\Big\|_{L^r(w)}^r
\end{equation}
Knowing that the desired family of estimate holds with some particular exponent, $q=r$, we want to invoke Rubio de Francia's extrapolation theorem with a new family $\mathcal{F}_r$ consisting of the pairs $(F,G)$, where
\[
F=\Big(\sum_j f_j^r\Big)^{\frac1r}, \qquad
G=\Big(\sum_j g_j^r\Big)^{\frac1r},
\]
and $\{(f_j,g_j)\}_j\subset\mathcal{F}$. In this way, \eqref{eq:wregrsbrtg} can be written as 
\begin{equation}\label{eq:ewreavce:new}
\|F\|_{L^r(w)}
\lesssim
\|G\|_{L^r(w)}
,\quad \text{for every }w\in A_r \text{ and }(F,G)\in\mathcal{F}_r.
\end{equation}
Therefore, we can extrapolate and show that 
\[
\|F\|_{L^q(w)}
\lesssim
\|G\|_{L^q(w)}
,\quad \text{for every }q\in(1,\infty), \text{ every }w\in A_q \text{ and }(F,G)\in\mathcal{F}_r.
\]
That is, \eqref{eq:ewreavce} holds as desired.

The previous application of Rubio de Francia's extrapolation theorem shows that, for one reason or another, sometimes there is a natural choice of an exponent for which the  desired estimate turns out to be simpler. The goal of this paper is to show that this applies to the case of weighted tent spaces which are introduced next. Given $0<r,p<\infty$ and $w\in A_\infty$ we say that $F\in L^r_\loc(\ree_+)$ belongs to $\mathscr{T}_r^q(w)$ provided
\[
\|F\|_{\mathscr{T}_r^p(w)}
:=
\|\mathcal{A}_r F\|_{L^p(w)}
:
=
\Big\|\Big( \iint_{|\,\cdot\,-y|<t} |F(y,t)|^r\,\frac{dy\,dt}{t^{n+1}}\Big)^{\frac1r}
\Big\|_{L^p(w)}<\infty.
\]
Much as in the case of the vector-valued inequalities estimates considered above, the choice $p=r$ allows us to use Fubini's theorem to rearrange the different integrals. This is the basic idea to obtain the main result of this paper which is formulated next. We also prove results in the weighted Lorentz tent spaces
$\mathscr{T}_r^{q,s}(w)$, with $0<r,p, s<\infty$ and $w\in A_\infty$, defined using the quasi-norm
\[
\|F\|_{\mathscr{T}_r^{p,s}(w)}
:=
\|\mathcal{A}_r F\|_{L^{p,s}(w)}
=
\Big\|\Big( \iint_{|\,\cdot\,-y|<t} |F(y,t)|^r\,\frac{dy\,dt}{t^{n+1}}\Big)^{\frac1r}
\Big\|_{L^{p,s}(w)}.
\]

\begin{theorem}\label{theor:extrapol}
Let $\mathcal{F}$ be a family of pairs of measurable functions $(F, G)$, with  $F, G:\ree_+\to [0,\infty]$, and so that $F(\cdot,t)$, $G(\cdot,t)$ are measurable for each $t\in (0,\infty)$. For any $t\in(0,\infty)$ introduce $\mathcal{F}_t$, the family of pairs $(F(\cdot,t),G(\cdot,t))$ with $(F,G)\in \mathcal{F}$. Let $\varphi:[1,\infty)\to [1,\infty)$ be a non-decreasing function.
\begin{list}{\textup{(\theenumi)}}{\usecounter{enumi}\leftmargin=.8cm \labelwidth=.8cm \itemsep=0.2cm \topsep=.2cm \renewcommand{\theenumi}{\roman{enumi}}}
	
\item Assume that there exists $p_0\in [1,\infty)$ such that for every $w\in A_{p_0}$ there holds
\begin{equation}\label{extrapol-hyp}
	\|F(\cdot,t)\|_{L^{p_0}(w)}
\le
\varphi([w]_{A_{p_0}})\,\|G(\cdot,t)\|_{L^{p_0}(w)},
\qquad (F(\cdot,t),G(\cdot,t))\in\mathcal{F}_t,\ t\in (0,\infty).
\end{equation}
For every $r,p\in (1,\infty)$ there exists a non-decreasing function $\psi:[1,\infty)\to [1,\infty)$ such that 
\begin{equation}\label{extrapol-conc}
\|F\|_{\mathscr{T}_r^p(w)}
\le
\psi([w]_{A_p})\,\| G\|_{\mathscr{T}_r^p(w)},
\qquad (F,G)\in\mathcal{F},
\end{equation}
for every $w\in A_{p}$.

\item Assume that there exists $p_0\in (0,\infty)$ such that for every $w\in A_{\infty}$ there holds
\[
\|F(\cdot,t)\|_{L^{p_0}(w)}
\le
\varphi([w]_{A_\infty})\,\|G(\cdot,t)\|_{L^{p_0}(w)},
\qquad (F(\cdot,t),G(\cdot,t))\in\mathcal{F}_t,\ t\in (0,\infty).
\]
For every $r,p,s\in (0,\infty)$ there exists a non-decreasing function $\psi:[1,\infty)\to [1,\infty)$ such that
\[
\|F\|_{\mathscr{T}_r^p(w)}
\le
\psi([w]_{A_\infty})\,\|G\|_{\mathscr{T}_r^p(w)},
\qquad (F,G)\in\mathcal{F},
\]	
and
\[
\|F\|_{\mathscr{T}_r^{p,s}(w)}
\le
\psi([w]_{A_\infty})\,\|G\|_{\mathscr{T}_r^{p,s}(w)},
\qquad (F,G)\in\mathcal{F},
\]	
for every $w\in A_{\infty}$.

\item Assume that there exist $0<p_-<p_+< \infty$, and $p_0\in [p_-,p_+]$,  such that for every $w\in A_{\frac{p_0}{p_-}}\cap RH_{(\frac{p_+}{p_0})'}$ there holds
\[
\|F(\cdot,t)\|_{L^{p_{0}}(w)}
\le
\varphi([w]_{A_{\frac{p_0}{p_-}}}+[w]_{RH_{(\frac{p_+}{p_0})'}})\, \|G(\cdot,t)\|_{L^{p_0}(w)},
\qquad (F(\cdot,t),G(\cdot,t))\in\mathcal{F}_t,\ t\in (0,\infty).
\]
For every $r,p\in (p_-,p_+)$ there exists a non-decreasing function $\psi:[1,\infty)\to [1,\infty)$ such that 
\[
\|F\|_{\mathscr{T}_r^p(w)}
\le
\psi([w]_{A_{\frac{p}{p_-}}}+[w]_{RH_{(\frac{p_+}{p})'}})\,\|G\|_{\mathscr{T}_r^p(w)},
\qquad (F,G)\in\mathcal{F},
\]
for every $w\in A_{\frac{p}{p_-}}\cap RH_{(\frac{p_+}{p})'}$.

\item  Assume that there exist $1\le p_0\le q_0<\infty $, such that for every $w\in A_{p_0,q_0}$ there holds
\[
\|F(\cdot,t)\|_{L^{q_0}(w^{q_0})}
\le
\varphi([w]_{A_{p_0,q_0}})\, \|G(\cdot,t)\|_{L^{p_0}(w^{p_0})},
\qquad (F(\cdot,t),G(\cdot,t))\in\mathcal{F}_t,\ t\in (0,\infty).
\]
For every $p$ and $q$ such that $1< p\le q<\infty$ and $\frac1p-\frac1q=\frac1{p_0}-\frac1{q_0}$, and for every $r\in (1,\infty)$
there exists a non-decreasing function $\psi:[1,\infty)\to [1,\infty)$ such that 
\[
\|F\|_{\mathscr{T}_r^q(w^q)}
\le
\psi([w]_{A_{p,q}})\,\|G\|_{\mathscr{T}_r^p(w^p)},
\qquad (F,G)\in\mathcal{F},
\]
for every $w\in A_{p,q}$.

\end{list}
\end{theorem}

In many applications, the function $F$ is of the form $F(t,\cdot) = T_{t}(G(t,.))$ for some family of sublinear operators $(T_{t})_{t >0}$.  Even when, for all $t>0$, $T_{t}=T$ for a fixed operator $T$ bounded on $L^{p}(\re^n)$, the estimate
\begin{equation}
\label{eq:mult}
\Big\|\Big( \iint_{|\,\cdot\,-y|<t} |T(G(t,\cdot))(y)|^r\,\frac{dy\,dt}{t^{n+1}}\Big)^{\frac1r}
\Big\|_{L^p}
\lesssim
\Big\|\Big( \iint_{|\,\cdot\,-y|<t} |G(y,t)|^r\,\frac{dy\,dt}{t^{n+1}}\Big)^{\frac1r}
\Big\|_{L^p},
%
\end{equation}
is non-trivial. This is somewhat surprising, given that the analogous estimate 
on the vertical square function space $L^{p}(\R^{n};L^{2}(\R_{+},\frac{dt}{t}))$, rather than the conical square function space $\mathscr{T}_2^p$, holds for simple reasons (e.g., by the classical result of Marcinkiewicz and Zygmund \cite{MZ}). 
To prove \eqref{eq:mult}, or its more general version with the family $(T_{t})_{t >0}$, one can use a (different) extrapolation method based on off-diagonal estimates of sufficiently high order decay, systemically developed in \cite{AKMP} (see Section \ref{section:od}). Unfortunately, some natural operators such as the Riesz transforms do not have enough off-diagonal decay and one can use weighted extrapolation instead, as first demonstrated in \cite{AP}. 

Theorem \ref{theor:extrapol} improves the approach developed in \cite{AP} by proving a general result that extrapolates inequalities, rather than just operators, in the spirit of \cite{CMP}. This adds flexibility, and allows us to add weights to the tent space norms, simplify proofs, and remove some unnecessary restrictions (see the applications to fractional operators in Section \ref{section:app}). We also prove, in Section \ref{section:od}, an extrapolation result in weighted tent spaces under an off-diagonal estimate assumption, and clarify the relationship between extrapolation results using off-diagonal assumptions and those using weighted estimates as a hypothesis.

Our approach is based on Lemma \ref{lemma:aver_w}, which describes how the averaging operators used to define tent spaces interact with weights. This is more flexible than the approach used in \cite{AP}, which is based on an understanding of how the averaging operators interact with the operators one wants to extrapolate (see \cite[Lemma 4.1, Lemma 5.4]{AP}).

Given Lemma \ref{lemma:aver_w}, our assumptions then translate into an estimate in some weighted Lebesgue spaces for some particular exponent and for the pairs $(\mathcal{A}_r F, \mathcal{A}_r G)$, with $(F,G)\in \mathcal{F}$. For the sake of specificity let us sketch the proof of (i). Fix $r\in(1,\infty)$. We will show that Fubini's theorem implies that for every $w\in A_r$ one has
\[
\|\mathcal{A}_r F\|_{L^r(w)}
\lesssim
\|\mathcal{A}_r G\|_{L^r(w)}, \qquad (F,G)\in\mathcal{F}.
\]
If we then define the family $\widetilde{\mathcal{F}}_r$ consisting of the pairs $(\mathcal{A}_r F, \mathcal{A}_r G)$ with $(F,G)\in\mathcal{F}$, we can extrapolate from the exponent $r$ to obtain estimates for $\widetilde{\mathcal{F}}_r$ in $L^p(w)$ for every $p\in(1,\infty)$ and for all $w\in A_p$. These are, in turn, the desired estimates for the pairs of the family $\mathcal{F}$ in $\mathscr{T}_r^p(w)$.

As just shown for (i), one can relate the initial estimates with some weighted norm inequality for the pairs $(\mathcal{A}_r F, \mathcal{A}_r G)$ with $(F,G)\in\mathcal{F}$. This means that we can invoke some other known extrapolation results and obtain other estimates. For instance, we can consider tent spaces associated to Banach functions space and in the context of (i) show that 
\[\|\mathcal{A}_r F\|_{\mathbb{X}(w)}
\lesssim
\|\mathcal{A}_r G\|_{\mathbb{X}(w)}
\]
with $\mathbb{X}$ being a rearrangement invariant Banach function space with Boyd indices $1<p_{\mathbb{X}}\le q_{\mathbb{X}}<\infty$ and $w\in A_{p_{\mathbb{X}}}$ (see \cite[Theorem~4.6]{CMP}). 
Examples to which this applies are $L^{p,q}$, $1<p<\infty$, $1\le q\le\infty$; or $L^p(\log L)^\alpha$ with $1<p<\infty$ and $\alpha\in\re$. In both cases $p_{\mathbb{X}}=q_{\mathbb{X}}=p$ and hence the weights should be in $A_p$. Other examples are $\mathbb{X}=L^p\cap L^q$ or $L^p+L^q$ with $1<p<q<\infty$, where $p_{\mathbb{X}}=p$ and $q_{\mathbb{X}}=q$, and the weights are assumed to be in $A_p$. 

More generally we can introduce some different weights inside the previous norms and obtain
\[
\|(\mathcal{A}_r F)\,u\|_{\mathbb{X}(w)}
\lesssim
\|(\mathcal{A}_r G)\,u\|_{\mathbb{X}(w)}
\]
with $\mathbb{X}$ being a rearrangement invariant Banach function space with Boyd indices $1<p_{\mathbb{X}}\le q_{\mathbb{X}}<\infty$, $u^{p_{\mathbb{X}}}\,w\in A_{p_{\mathbb{X}}}$, $u^{q_{\mathbb{X}}}\,w\in A_{q_{\mathbb{X}}}$, and $w\in A_\infty$ (see \cite[Theorem~3.1, Example~2.46]{CMM}). For instance, if $\mathbb{X}=L^{p,q}$, $1<p<\infty$, $1\le q\le\infty$; or $L^p(\log L)^\alpha$ with $1<p<\infty$ and $\alpha\in\re$
the conditions on the weights reduce to $u^{p}\,w\in A_{p}$ and $w\in A_\infty$. As before, for $\mathbb{X}=L^p\cap L^q$ or $L^p+L^q$ with $1<p<q<\infty$, the weights should satisfy $u^p\,w\in A_p$, $u^q\,w\in A_{q}$, and $w\in A_\infty$. 

We can also consider Banach function spaces which are not necessarily rearrangement invariant. Indeed, if  $\mathbb{X}$ is Banach function space so that the Hardy-Littlewood maximal function is bounded both on $\mathbb{X}$ and 
its associate space $\mathbb{X}'$ (cf. \cite[Theorem~8.2]{MMMMM}, \cite[Theorem~3.1]{CMM}), then we also get that  $\|\mathcal{A}_r F\|_{\mathbb{X}}
\lesssim \|\mathcal{A}_r G\|_{\mathbb{X}}$. That is the case of variable Lebesgue spaces $\mathbb{X}=L^{p(\cdot)}$ where the variable exponent $p(\cdot)$ satisfies the so-called log-Hölder condition (see \cite{CMM} and the references therein). One could also obtain estimates of the form $\|(\mathcal{A}_r F)\,u\|_{\mathbb{X}}
\lesssim \|(\mathcal{A}_r G)\,u\|_{\mathbb{X}}$ provided the Hardy-Littlewood maximal function and some appropriate dual operator satisfies similar inequalities (see \cite[Theorem~3.1]{CMM}). As a consequence,$\|(\mathcal{A}_r F)\,u\|_{L^{p(\cdot)}} \lesssim \|(\mathcal{A}_r G)\,u\|_{L^{p(\cdot)}}$ holds for variable exponents $p(\cdot)$ satisfying the log-Hölder condition and for all $u\in A_{p(\cdot)}$ (see \cite[Example~2.39]{CMM}).

 More general weighted Banach functions spaces are allowed, see \cite[Theorem~3.1]{CMM}, or even estimates in modular spaces, see \cite[Section~4.3]{CMP} or \cite[Theorem~4.1]{CMM}. The precise statements are left to the interested reader.

The plan of the paper is as follows. We present some applications of the main result in the next section:  Corollary~\ref{corol:main}  is a general result which is applied to Calder\'on-Zygmund operators, operators associated with the Kato conjecture, and fractional operators. Notation is fixed in Section \ref{section:prelim}, before we turn to the proofs in Section \ref{section:proof}. Finally, in Section \ref{section:od}, we consider the interaction between extrapolation results on tent spaces based on off-diagonal bounds, and those based on weighted estimates.

\section{Applications}
 \label{section:app}

We denote by $M$ the (uncentered) Hardy-Littlewood maximal function and by $\mathcal{M}$ its corresponding extension to functions defined in $\ree_+$. That is, if $F:\ree_+\to\re$ with $F(\cdot,t)\in L^1_\loc(\re^n)$ we set 
\[
\mathcal{M}(F)(x,t)
:=
M(F(\cdot,t))(x),
\qquad 
(x,t)\in\ree_+.
\]
Analogously, for an operator $T$ on $\re^n$ defined on some family of measurable functions we define its extension
\begin{equation}\label{ext-T}
	\mathcal{T}(F)(x,t)
:=
T(F(\cdot,t))(x),
\qquad 
(x,t)\in\ree_+,
\end{equation}
provided $F(\cdot,t)$ belongs to the domain of $T$.

The first consequence of Theorem~\ref{theor:extrapol} is that we can extend \cite[Theorem~2.1]{AP} to consider the boundedness of the Hardy-Littlewood maximal function in weighted tent spaces.

\begin{corollary}\label{corol:M-tent}
Given $1<r<\infty$, $\mathcal{M}$ is bounded on $\mathscr{T}_r^p(w)$ for all $1<p<\infty$ and all $w\in A_p$, and bounded from $\mathscr{T}_r^1(w)$ to $\mathscr{T}_r^{1,\infty}(w)$ for all $w\in A_1$.
\end{corollary}

The proof of the case $1<p<\infty$ is particularly easy. Let $\mathcal{F}$ be the family of pairs $(\mathcal{M}(G),|G|)$, where $G:\ree_+ \to\re$ with $G(\cdot,t)\in L^1_\loc(\re)$. Note that for every $t>0$ and for every $w\in A_2$ by \cite{Buckley}
\[
\|\mathcal{M}(G)(\cdot,t)\|_{L^2(w)}
=
\|M(G(\cdot,t))\|_{L^2(w)}
\lesssim
[w]_{A_2}\,\|G(\cdot,t)\|_{L^2(w)}.
\]
We can then invoke Theorem~\ref{theor:extrapol} part (i) with $p_0=2$ to obtain for every $p,r\in (1,\infty)$ and every $w\in A_p$ 
\[
\|\mathcal{M}(G)\|_{\mathscr{T}_r^p(w)}
\lesssim
\|G\|_{\mathscr{T}_r^p(w)}.
\]
The case $p=1$ does not follow by extrapolation and its proof is postponed until the end of Section \ref{section:proof}. 

Using the very same ideas, Theorem~\ref{theor:extrapol} and Corollary~\ref{corol:M-tent} have the following immediate consequence. The proof is easy and left to the interested reader.

\begin{corollary}\label{corol:main} \null\ \null
\begin{list}{\textup{(\theenumi)}}{\usecounter{enumi}\leftmargin=.8cm \labelwidth=.8cm \itemsep=0.2cm \topsep=.2cm \renewcommand{\theenumi}{\alph{enumi}}}

\item Assume that $T$ is an operator satisfying 
	\(
	\|T f\|_{L^{p_0}(w)}\lesssim \|f\|_{L^{p_0}(w)},
	\)
	for some $p_0\in (1,\infty)$ and for all $w\in A_{p_0}$ (equivalently, for all $p\in (1,\infty)$ and all $w\in A_p$). 
	Then, $\mathcal{T}$ is bounded in $\mathscr{T}_r^p(w)$ for all $p,r\in (1,\infty)$ and all $w\in A_p$, that is, 
	\[
	\|\mathcal{T}(F)\|_{\mathscr{T}_r^p(w)}	\lesssim\|F\|_{\mathscr{T}_r^p(w)}.
	\]

\item Assume that $T$ is an operator satisfying 
\(
\|T f\|_{L^{p_0}(w)}\lesssim \|M f\|_{L^{p_0}(w)},
\)
for some $p_0\in (0,\infty)$ and for all $w\in A_{\infty}$ (equivalently, for all $p\in (0,\infty)$ and all $w\in A_\infty$). 
Then, for all $p,r,s\in (0,\infty)$ and all $w\in A_\infty$, there hold
\[
\|\mathcal{T}(F)\|_{\mathscr{T}_r^p(w)}	\lesssim\|\mathcal{M}(F)\|_{\mathscr{T}_r^p(w)},
\qquad
\|\mathcal{T}(F)\|_{\mathscr{T}_r^{p,s}(w)}	\lesssim\|\mathcal{M}(F)\|_{\mathscr{T}_r^{p,s}(w)}.
\]
As a consequence, $\mathcal{T}$ is bounded in $\mathscr{T}_r^p(w)$ for all $p,r\in (1,\infty)$ and all $w\in A_p$, that is, 
\[
\|\mathcal{T}(F)\|_{\mathscr{T}_r^p(w)}
\lesssim
\|F\|_{\mathscr{T}_r^p(w)};
\]
and  $\mathcal{T}$ is bounded from $\mathscr{T}_r^1(w)$ to $\mathscr{T}_r^{1,\infty}(w)$  for all $r\in (1,\infty)$ and all $w\in A_1$, that is, 
\[
\|\mathcal{T}(F)\|_{\mathscr{T}_r^{1,\infty}(w)}
\lesssim
\|F\|_{\mathscr{T}_r^1(w)}.
\]

\item Assume that there exist $0< p_-<p_+< \infty$ and that $T$ is an operator satisfying 
\(
\|T f\|_{L^{p_0}(w)}\lesssim \|f\|_{L^{p_0}(w)},
\)
for some $p_0\in (p_-,p_+)$ and for all $w\in A_{{p_0}/{p_-}}\cap RH_{({p_+}/{p_0})'}$ (equivalently, for all $p\in (p_-,p_+)$ and all $w\in A_{{p}/{p_-}}\cap RH_{({p_+}/{p})'}$). 
Then, $\mathcal{T}$ is bounded in $\mathscr{T}_r^p(w)$ for all $p,r\in (p_-,p_+)$ and all $w\in A_{{p}/{p_-}}\cap RH_{({p_+}/{p})'}$, that is, 
\[
\|\mathcal{T}(F)\|_{\mathscr{T}_r^p(w)}	\lesssim\|F\|_{\mathscr{T}_r^p(w)}.
\]

\item  Assume that there exist $1< p_0\le q_0<\infty $ and that $T$ is an operator satisfying 
\(
\|T f\|_{L^{q_0}(w^{q_0})}\lesssim \|f\|_{L^{p_0}(w^{p_0})},
\)
for every $w\in A_{p_0,q_0}$ (equivalently, for all $1< p\le q<\infty$ with $\frac1p-\frac1q=\frac1{p_0}-\frac1{q_0}$ and all $w\in A_{p,q}$). Then, $\mathcal{T}$ is bounded from $\mathscr{T}_r^p(w^p)$ to $\mathscr{T}_r^q(w^q)$ for all 
$p$ and $q$ such that $1< p\le q<\infty$ and $\frac1p-\frac1q=\frac1{p_0}-\frac1{q_0}$, all $r\in (1,\infty)$, and all $w\in A_{p,q}$, that is, 
\[
\|\mathcal{T}(F)\|_{\mathscr{T}_r^q(w^q)}	\lesssim\|F\|_{\mathscr{T}_r^p(w^p)}.
\]
\end{list}
\end{corollary}

In the previous result all the involved boundedness should hold with a constant that depends on the characteristic of the weights in terms on a non-decreasing function. 

Let us now present some relevant examples of operators to which the previous result applies:

\subsection*{Calderón-Zygmund operators and Singular Integral operators}
Any Calderón-Zygmund operator $T$, like the Hilbert or the Riesz transform, is bounded on $L^p(w)$ for every $p\in(1,\infty)$ and $w\in A_p$. It also satisfies Coifmann-Fefferman inequalities (cf.~\cite{Coi, CF}), that is, $T$ is  controlled by $M$ in $L^p(w)$ for every $p\in(0,\infty)$ and $w\in A_\infty$. As such we can apply Corollary~\ref{corol:main} parts (a) and (b) to obtain that $\mathcal{T}$ is bounded in $\mathscr{T}_r^p(w)$ for every $1<p,r<\infty$ and every $w\in A_p$ and is also bounded from $\mathscr{T}_r^1(w)$ to $\mathscr{T}_r^{1,\infty}(w)$ for every $1<r<\infty$ and every $w\in A_1$, this extends \cite[Theorem 2.2]{AP}. Moreover, $\mathcal{T}$ can be controlled by $\mathcal{M}$ in
$\mathscr{T}_r^{p}(w)$, and even in $\mathscr{T}_r^{p,s}(w)$, for every $0<p,r,s<\infty$ and $w\in A_\infty$. This can be generalized to other singular integral operators. For instance we can consider commutators $[T,b]$ with $T$ being a Calderón-Zygmund operator and $b\in\mathrm{BMO}$. It is well-known that these commutators are bounded in $L^p(w)$ for $1<p<\infty$ and $w\in A_p$ and they can be controlled by $M^2=M\circ M$ in $L^p(w)$ for every $p\in(0,\infty)$ and $w\in A_\infty$. We can therefore obtain that these estimates readily extend to weighed tent spaces, further details are left to the interested reader.

\subsection*{Operators associated with the Kato conjecture}
Let $A$ be an $n \times n$ matrix of complex and $L^{\infty}$-valued coefficients defined on $\re^n$. We assume that $A$ satisfies the following ellipticity condition: there exist $0<\lambda \le \Lambda<\infty$ such that
\[
\lambda |\xi|^2 \le \mathrm{Re}\, A(x) \xi \cdot \bar{\xi} \qquad\text{and}\qquad 
|A(x)\xi \cdot \bar{\eta}| \le \Lambda |\xi| |\eta|, 
\]
for all $\xi, \eta \in \C^n$ and almost every $x \in \re^n$. We have used the notation $\xi \cdot \eta=\sum_{j=1}^n \xi_j \eta_j$, and therefore $\xi \cdot \bar{\eta}$ is the usual inner product in $\C^n$. Note that then $A(x)\xi \cdot \bar{\eta}=\sum_{j, k}a_{j,k}(x) \xi_k \bar{\eta}_j$. Associated with this matrix we define the second order divergence form operator $Lu =-\div(A \nabla u)$, which is understood in the standard weak sense as a maximal-accretive operator on the space $L^2(\re^n)$ with domain $D(L)$  by means of a sesquilinear form. 

Associated with this operator we can consider the functional calculus $\varphi(L)$ where $\varphi$ is holomorphic and bounded in an appropriate sector, the Riesz transform $\nabla L^{-1/2}$, and some square functions. The $L^p$ theory for these operators was developed in the monograph \cite{Aus} and weighted norm inequalities were obtained in \cite{AM3} using a generalized Calderón-Zygmund theory from \cite{AM1}.  As in \cite{Aus} and \cite{AM2}, we denote by $(p_{-}(L), p_{+}(L))$, respectively  $(q_{-}(L), q_{+}(L))$, the maximal open interval on which the heat semigroup $(e^{-tL})_{t>0}$, respectively its gradient $(\sqrt{t}\nabla e^{-tL})_{t>0}$,  is uniformly bounded on $L^p(\re^n)$. It is obtained in \cite{Aus} that $p_-(L)=q_-(L)$ and $2<q_+(L)\le p_+(L)$. 

Recall the weighted norm inequalities from \cite{AM3}: 
\begin{align*}
	\|\varphi(L) f\|_{L^p(w)} \lesssim \|\varphi\|_{\infty} \|f\|_{L^p(w)},\quad \forall\,p \in (p_{-}(L), p_{+}(L)), \, \, 
	w \in A_{p/p_{-}(L)} \cap RH_{(p_{+}(L)/p)'};
\end{align*}
\begin{align*}
	\|\nabla L^{-1/2} f\|_{L^p(w)} \lesssim \|f\|_{L^p(w)}, \quad \forall\,p \in (q_{-}(L), q_{+}(L)), \, \, 
	w \in A_{p/q_{-}(L)} \cap RH_{(q_{+}(L)/p)'}; 
\end{align*}
and
\begin{align*}
	\|L^{1/2} f \|_{L^p(w)} \lesssim \|\nabla f\|_{L^p(w)},\quad \forall\,p \in (p_{-}(L), p_{+}(L)), \, \, 
	w \in A_{p/p_{-}(L)} \cap RH_{(p_{+}(L)/p)'}. 
\end{align*}
Using these estimates, and invoking Corollary~\ref{corol:main} part (c) we can easily obtain that $\varphi(L)$ can be extended as a bounded operator in $\mathscr{T}_r^p(w)$ for every $p,r \in (p_{-}(L), p_{+}(L))$ and $w \in A_{p/p_{-}(L)} \cap RH_{(p_{+}(L)/p)'}$. In the case of the Riesz transform $\nabla L^{-1/2}$ we get estimates in $\mathscr{T}_r^p(w)$ for every $p,r \in (q_{-}(L), q_{+}(L))$ and $w \in A_{p/q_{-}(L)} \cap RH_{(q_{+}(L)/p)'}$. The latter provides a much easier proof of \cite[Theorem~2.4]{AP}, which only considers the unweighted situation.

Finally for the reverse inequalities for square roots we get 
\begin{align*}
	\|L^{1/2} F \|_{\mathscr{T}_r^p(w)} \lesssim \|\nabla_x F\|_{\mathscr{T}_r^p(w)},\quad \forall\,p,r \in (p_{-}(L), p_{+}(L)), \, \, 
	w \in A_{p/p_{-}(L)} \cap RH_{(p_{+}(L)/p)'}. 
\end{align*}
In particular we get the following Kato type estimates in weighted tent spaces:
\begin{align*}
	\|L^{1/2} F \|_{\mathscr{T}_r^p(w)} \approx \|\nabla_x F\|_{\mathscr{T}_r^p(w)},\quad \forall\,p,r \in (q_{-}(L), q_{+}(L)), \, \, 
	w \in A_{p/q_{-}(L)} \cap RH_{(q_{+}(L)/p)'}. 
\end{align*}
In these estimates it is understood that the operators in question are acting only on the space variable (much as in \eqref{ext-T}), but this time, abusing the notation, we have not used a different symbol to denote this extension.

\subsection*{Fractional operators}
Let $I_\alpha$ denote the fractional integral (aka Riesz potential) of order $\alpha\in (0,n)$ and let $M_\alpha$ be the associated maximal fractional integral operator. We write $\mathcal{I}_\alpha$ and $\mathcal{M}_\alpha$ for the associated extensions to $\ree_+$ (cf.~\eqref{ext-T}). In \cite[Theorem 1]{MW} it was shown that $I_\alpha$ is controlled by $M_\alpha$ in $L^p(w)$ for all $p\in (0,\infty)$ and $w\in A_\infty$. By Corollary~\ref{corol:main} part $(b)$ we readily obtain that $\mathcal{I}_\alpha$ is controlled by $\mathcal{M}_\alpha$ in the weighted tent spaces $\mathscr{T}_r^p(w)$, for every $0<p,r<\infty$ and $w\in A_\infty$. On the other hand, \cite{MW} also shows that $I_\alpha$ and $M_\alpha$ are bounded from $L^p(w^p)$ to $L^q(w^q)$ for every $p\in(1, n/\alpha)$, $q$ so that $\frac1p-\frac1q=\frac\alpha{n}$, and $w\in A_{p,q}$. Invoking Corollary~\ref{corol:main} part (d) we conclude that  $\mathcal{I}_\alpha$ and $\mathcal{M}_\alpha$ are bounded in the weighted tent spaces $\mathscr{T}_r^p(w)$ for every $p\in(1, n/\alpha)$, $q$ so that $\frac1p-\frac1q=\frac\alpha{n}$, $r\in (1,\infty)$, and $w\in A_{p,q}$. This extends \cite[Theorem 2.3]{AP} which only considers the unweighted case and where it was additionally assumed that $r>n/(n-\alpha)$.

\section{Notation}\label{section:prelim}

Throughout this paper, we work in $\re^n$, $n\ge 1$, endowed with the Lebesgue measure which is denoted by $dx$. The characteristic function of a measurable set $E$ is denoted by $\mathbf{1}_E$, and we use $B$ to denote Euclidean balls. The Hardy-Littlewood maximal operator $M$ is defined for each locally integrable function $f$ by 
\begin{equation}\label{eq:M-def}
	M f(x) := \sup_{x\in B} \fint_B |f(y)|\,dy, \qquad x\in\re^n, 
\end{equation} 
where the sup runs over all Euclidean balls $B$ containing $x$. Here and elsewhere barred integrals are used to denote averages. 

A measurable function $w$ on $\re^n$ is called a weight if $0<w(x)<\infty$ for a.e.~$x \in \re^n$. For every $p\in(0,\infty)$ and every weight $w$, we define the associated weighted Lebesgue space $L^p(w):=L^p(w(x)\,dx)$ and the same applies to weighted Lorentz spaces $L^{p,s} (w):=L^{p,s}(w(x)\,dx)$ for $p,s\in (0,\infty)$. 

Recall that we define tent spaces as follows: given $0<r,p<\infty$ and a weight $w$ we say that $F\in L^r_\loc(\ree_+)$ belongs to $\mathscr{T}_r^q(w)$ provided
\[
\|F\|_{\mathscr{T}_r^p(w)}
:=
\|\mathcal{A}_r F\|_{L^p(w)}
:
=
\Big\|\Big( \iint_{|\,\cdot\,-y|<t} |F(y,t)|^r\,\frac{dy\,dt}{t^{n+1}}\Big)^{\frac1r}
\Big\|_{L^p(w)}<\infty.
\]
Weighted Lorentz tents spaces are defined similarly.

Given $p \in (1, \infty)$, we define the class $A_p$ as the collection of all weights $w$ satisfying 
\begin{equation*}
	[w]_{A_p}:=\sup_{B} \left(\fint_{B} w(x)\, dx\right) \left(\fint_{B}w(x)^{1-p'}\, dx \right)^{p-1}<\infty,
\end{equation*} 
where $p'$ is the H\"older conjugate exponent of $p$, i.e., $\frac1p+\frac{1}{p'}=1$. As for the case $p=1$, we say that $w\in A_{1}$ if  
\begin{equation*}
	[w]_{A_1} := \sup_{B} \left(\fint_{B} w(y)\, dy\right) \|w^{-1}\mathbf{1}_{B}\|_{L^{\infty}(\re^n)}
	<\infty.
\end{equation*}
Finally, we define 
\begin{equation*}
	A_{\infty} :=\bigcup_{p\geq 1}A_{p},
\end{equation*} 
and $[w]_{A_\infty}=\inf_{p:w \in A_{p}} [w]_{A_{p}}$.

Given $1<p\le\infty$, $1 \leq q < \infty$, we define the class $A_{p,q}$ as the collection of all weights $w$ satisfying
\begin{equation}\label{Apq}
	[w]_{A_{p, q}} := \sup_{B } \left(\fint_{B}w(x)^q\, dx \right) \left(\fint_{B}w(x)^{-p'}\, dx\right)^{\frac{q}{p'}}<\infty.
\end{equation} 
The reverse Hölder classes are defined in the following way: we say that $w\in RH_{s}$ for $s\in(1,\infty)$  if 
\begin{equation*}
	[w]_{RH_{s}} :=\sup_{B } \left(\fint_B w(x)^s\,dx\right)^{\frac1s} \left(\fint_B w(x)\,dx\right)^{-1} < \infty. 
\end{equation*} 
Regarding the endpoint $s=\infty$, $w \in RH_{\infty}$ means that
\begin{equation*}
	[w]_{RH_{\infty}} := \sup_{B} \|w\mathbf{1}_{B}\|_{L^{\infty}(\re^n)} \left(\fint_{B}w(x)\,dx\right)^{-1}<\infty.  
\end{equation*}

\section{Proof of the main result}\label{section:proof}

We begin with some auxiliary results.

\begin{lemma}\label{lemma:aver}
For every $0\le h\in L^1_\loc(\re^n)$, $x\in\re^n$, and $s,t\in (0,\infty)$ there holds
\[
\fint_{B(x,s)} \Big(\fint_{B(y,t)} h(z)\,dz\Big)\,dy
\le
2^n
\fint_{B(x,s+t)} h(y)\,dy.
\]
\end{lemma}

\begin{proof}
Note first that if $y\in B(x,s)$ then $B(y,t)\subset B(x, s+t)$. Assume first that $s\le t$. Then,
\begin{multline*}
\fint_{B(x,s)} \Big(
\fint_{B(y,t)} h(z)\,dz\Big)\,dy
\le
\fint_{B(x,s)} \Big(\frac{|B(x,s+t)|}{|B(y,t)|}\,\fint_{B(x,s+t)} h(z)\,dz\Big)\,dy
\\
=
\frac{(s+t)^n}{t^n}\,\fint_{B(x,s+t)} h(z)\,dz\le
2^n
\fint_{B(x,s+t)} h(y)\,dy.
\end{multline*}
On the other hand, if $t\le s$, we observe that Fubini's theorem implies
\begin{multline*}
\fint_{B(x,s)} \Big(
\fint_{B(y,t)} h(z)\,dz\Big)\,dy
=
\fint_{B(x,s)} \Big(
\fint_{B(y,t)} h(z)\,\mathbf{1}_{B(x,s+t)}\,dz\Big)\,dy
\\
\le 
\fint_{B(x,s+t)} h(z)\,\Big(\fint_{B(z,t)}\frac{|B(x,s+t)|}{|B(x,s)|}\, dy\Big)\,dz
=
\frac{(s+t)^n}{s^n}\,\fint_{B(x,s+t)} h(z)\,dz\\
\le
2^n
\fint_{B(x,s+t)} h(y)\,dy.
\end{multline*}
This completes the proof.
\end{proof}

\begin{lemma}\label{lemma:aver_w}
Let $1\le p,r<\infty$ and $w\in A_p$. Define
\[
W_t(x)
:=
\fint_{B(x,t)} w(y)\,dy,
\qquad
x\in\re^n,\ t>0.
\]
Then, $W_t\in A_p$ with $[W_t]_{A_p}\le 2^{n\,p}\,[w]_{A_p}$, and for every $F\in L^r_\loc(\ree)$ 
\[
\|\mathcal{A}_r F\|_{L^r(w)}^r
=
v_n\int_0^\infty \|F(\cdot,t)\|_{L^r(W_t)}^r\,\frac{dt}{t},
\]
where $v_n=|B(0,1)|$.
\end{lemma}

\begin{proof}
Consider first the case $p=1$. Given $B=B(x,s)$, by Lemma~\ref{lemma:aver} we obtain
\begin{multline*}
\fint_B W_t(y)\,dy
=
\fint_{B(x,s)} \Big(\fint_{B(y,t)} w(z)\,dz\Big)\,dy
\le
2^n\,\fint_{B(x,s+t)} w(y)\,dy
\\
\le
2^n\,[w]_{A_1}\essinf_{B(x,s+t)} w
\le
2^n\,[w]_{A_1}\essinf_{B(x,s)} w
.
\end{multline*}
This readily implies that $W_t\in A_1$ with $[W_t]_{A_1}\le 2^{n}\,[w]_{A_1}$.

Assume next that $1<p<\infty$. Using that $\eta(t)=t^{1-p'}$, $t>0$, is convex; Jensen's inequality; and Lemma~\ref{lemma:aver} we arrive at
\begin{align*}
&\Big(\fint_B W_t(y)\,dy\Big)\Big(\fint_B W_t(y)^{1-p'}\,dy\Big)^{p-1}
\\
&\qquad\qquad=
\Big(\fint_{B(x,s)} \Big(\fint_{B(y,t)} w(z)\,dz\Big)\,dy\Big)
\,
\Big(\fint_{B(x,s)} \Big(\fint_{B(y,t)} w(z)\,dz\Big)^{1-p'}\,dy\Big)^{p-1}
\\
&\qquad\qquad\le
\Big(\fint_{B(x,s)} \Big(\fint_{B(y,t)} w(z)\,dz\Big)\,dy\Big)
\,
\Big(\fint_{B(x,s)} \Big(\fint_{B(y,t)} w(z)^{1-p'}\,dz\Big)\,dy\Big)^{p-1}
\\
&\qquad\qquad
\le
2^{n\,p}
\Big(\fint_{B(x,s+t)} w(y)\,dy\Big)
\,
\Big(\fint_{B(x,s+t)} w(y)^{1-p'}\,dy \Big)^{p-1}
\\
&\qquad\qquad\le
2^{n\,p} [w]_{A_p}.
\end{align*}
From here we can immediately see that $W_t\in A_p$ with $[W_t]_{A_p}\le 2^{n\,p}\,[w]_{A_p}$.

Finally, Fubini's theorem readily gives that
\begin{multline*}
	\|\mathcal{A}_r F\|_{L^r(w)}^r
	=
	\int_{\ree} \iint_{|x-y|<t} |F(y,t)|^r\,\frac{dy\,dt}{t^{n+1}}\,w(x)\,dx
	=
	\int_0^\infty\!\!\!\int_{\re^n} |F(y,t)|^r\,\Big(\frac1{t^n} \int_{B(y,t)} w(x)\,dx\Big)\,dy\,\frac{dt}{t}
	\\
	=
	v_n\,\int_0^\infty\!\!\!\int_{\re^n} |F(y,t)|^r\,W_t(y)\,dy\,\frac{dt}{t}
	=
	v_n\int_0^\infty \|F(\cdot,t)\|_{L^r(W_t)}^r\,\frac{dt}{t},
\end{multline*}
and this completes the proof. 
\end{proof}

We are now ready to prove Theorem~\ref{theor:extrapol}.

\begin{proof}[Proof of Theorem~\ref{theor:extrapol}]
We start with (i). Fix $r\in (1,\infty)$ and let $w\in A_r$ be arbitrary. By \eqref{extrapol-hyp} and \cite[Theorem~3.9]{CMP} it follows that \eqref{extrapol-hyp} holds with
$r$ in place of $p_0$ and with some non-decreasing function $\widetilde{\varphi}$ replacing $\varphi$. This, together with Lemma~\ref{lemma:aver_w}, gives
\begin{multline*}
	\|\mathcal{A}_r F\|_{L^r(w)}^r
=
v_n\int_0^\infty \|F(\cdot,t)\|_{L^r(W_t)}^r\,\frac{dt}{t}
\le
v_n
\int_0^\infty \widetilde{\varphi}([W_t]_{A_r})\|G(\cdot,t)\|_{L^r(W_t)}^r\,\frac{dt}{t}
\\
\le
v_n \widetilde{\varphi}(2^n[w]_{A_r})\,\int_0^\infty\|G(\cdot,t)\|_{L^r(W_t)}^r\,\frac{dt}{t}
=
\widetilde{\varphi}(2^n[w]_{A_r})\,\|\mathcal{A}_r G\|_{L^r(w)}^r.
\end{multline*}
Since $w\in A_r$ is arbitrary we can extrapolate invoking again \cite[Theorem~3.9]{CMP}  ---with the family of pairs $(\mathcal{A}_r F,\mathcal{A}_r G)$, where $(F,G)\in\mathcal{F}$--- to conclude that \eqref{extrapol-conc} holds.

The proof of (ii) and (iii) can be carried out similarly invoking \cite[Corollary~3.15 and Corllary 3.17]{CMP} and \cite[Theorem~2.1]{CMP:Ainfty}, and \cite[Theorem 3.31]{CMP} respectively (see also \cite[Theorem~4.9]{AM1}). Further details are left to the reader. 

Proving (iv) requires some extra arguments that are inspired by \cite[p.~49]{CMP}. Fix $r\in (1,\infty)$ and pick  $p_1$ and $q_1$ such that $1< p_1\le r \le q_1<\infty$ and $\frac1{p_1}-\frac1{q_1}=\frac1{p_0}-\frac1{q_0}$ (this is possible for all values of $r\in (1,\infty)$). Take an arbitrary $w\in A_{p_1,q_1}$ and $(F,G)\in\mathcal{F}$. Note that by Minkowski's inequality with exponent $q_1/ r\ge 1$ and Jensen's inequality we arrive at
\begin{multline}\label{qfvevar}
\|\mathcal{A}_r F\|_{L^{q_1}(w^{q_1})}^r
=
v_n\,\Big\| 
\int_0^\infty \Big(\fint_{B(\cdot,t)} |F(y,t)|^r\,dy\Big)\,\frac{dt}{t}\Big\|_{L^{\frac{q_1}{r}}(w^{q_1})}
\\
\le
v_n\,\int_0^\infty \Big\| \fint_{B(\cdot,t)} |F(y,t)|^r\,dy\Big\|_{L^{\frac{q_1}{r}}(w^{q_1})}\,\frac{dt}{t}
\le
v_n\,\int_0^\infty \Big\|\Big(\fint_{B(\cdot,t)} |F(y,t)|^{q_1}\,dy\Big)^\frac{r}{q_1}\Big\|_{L^{\frac{q_1}{r}}(w^{q_1})}\,\frac{dt}{t}.
\end{multline}

Define \[
\widetilde{W}_t(x)
:=
\Big(\fint_{B(x,t)} w(y)^{q_1}\,dy\Big)^{\frac1{q_1}},
\qquad
x\in\re^n,\ t>0.
\]
Using that $w\in A_{p_1,q_1}$ if and only if $w^{q_1}\in A_{1+q_1/p_1'}$, and Lemma~\ref{lemma:aver_w} we readily obtain that 
$\widetilde{W}_t^{q_1}\in A_{1+q_1/p_1'}$ with $[\widetilde{W}_t^{q_1}]_{A_{1+q_1/p_1'}}\le 2^{n\,(1+q_1/p_1')}\,[w^{q_1}]_{A_{1+q_1/p_1'}}$. Equivalently, $\widetilde{W}_t\in A_{p_{1},q_{1}}$ with $[\widetilde{W}_t]_{A_{p_1,q_1}}\le 2^{n\,(1/q_1+1/p_1')}\,[w]_{A_{p_{1},q_{1}}}$. On the other hand, using that $w\in A_{p_{1},q_{1}}$; that $\eta(t)=t^{-{p_1}/{p_1'}}$, $t>0$, is convex; and Jensen's inequality we obtain
\[
\widetilde{W}_t(x)^{p_1}
=
\Big(\fint_{B(x,t)} w(y)^{q_1}\,dy\Big)^{\frac{^{p_1}}{q_1}}
\le
[w]_{A_{p_{1},q_{1}}}^{p_1}\,\Big(\fint_{B(x,t)} w(y)^{-p_1'}\,dy\Big)^{-\frac{{p_1}}{p_1'}}
\le
[w]_{A_{p_{1},q_{1}}}^{p_1}\,\fint_{B(x,t)} w(y)^{p_1}\,dy.
\]
All this give for any $t\in (0,\infty)$ by Fubini's inequality and our hypothesis
\begin{align*}
\Big\|\Big(\fint_{B(\cdot,t)} |F(y,t)|^{q_1}\,dy\Big)^\frac{r}{q_1}\Big\|_{L^{\frac{q_1}{r}}(w^{q_1})}^{\frac{1}{r}}
&=
\Big(\int_{\re^n}\Big(\fint_{B(x,t)} |F(y,t)|^{q_1}\,dy\Big)\,w(x)^{q_1}\,dx\Big)^{\frac1{q_1}}
\\
&=
\Big(\int_{\re^n}|F(y,t)|^{q_1}\,\Big(\fint_{B(y,t)} w(x)^{q_1}\,dx\Big)\,dy\Big)^{\frac1{q_1}}
\\
&
=
\|F(\cdot,t)\|_{L^{q_1}(\widetilde{W}_t^{q_1})}
\\
&
\lesssim
\|G(\cdot,t)\|_{L^{p_1}(\widetilde{W}_t^{p_1})}
\\
&\lesssim
\Big(\int_{\re^n}|G(y,t)|^{p_1}\,\Big(\fint_{B(y,t)} w(x)^{p_1}\,dx\Big)\,dy\Big)^{\frac1{p_1}}
\\
&
=
\Big(\int_{\re^n}\Big(\fint_{B(x,t)} |G(y,t)|^{p_1}\,dy\Big)\,w(x)^{p_1}\,dx\Big)^{\frac1{p_1}}
.
\end{align*}
Plugging this into \eqref{qfvevar}, using Minkowski's inequality with exponent $r/p_1\ge 1$ and Jensen's inequality we conclude that
\begin{multline*}
\|\mathcal{A}_r F\|_{L^{q_1}(w^{q_1})}^{p_1}
\lesssim
\Big(\int_0^\infty \Big(\int_{\re^n}\Big(\fint_{B(x,t)} |G(y,t)|^{p_1}\,dy\Big)\,w(x)^{p_1}\,dx\Big)^{\frac{r}{p_1}}\,\frac{dt}{t}\Big)^{\frac{p_1}{r}}
\\
\le
\int_{\re^n}\Big(\int_0^\infty \Big(\fint_{B(x,t)} |G(y,t)|^{p_1}\,dy\Big)^{\frac{r}{p_1}}\,\frac{dt}{t}\,\Big)^{\frac{p_1}{r}}\,w(x)^{p_1}\,dx
\\
\le
\int_{\re^n}\Big(\int_0^\infty \Big(\fint_{B(x,t)} |G(y,t)|^{r}\,dy\Big)\,\frac{dt}{t}\,\Big)^{\frac{p_1}{r}}\,w(x)^{p_1}\,dx
=
v_n^{-\frac{p_1}{r}}\|\mathcal{A}_r G\|_{L^{p_1}(w^{p_1})}^{p_1}.
\end{multline*}
All these show that $\|\mathcal{A}_r F\|_{L^{q_1}(w^{q_1})}\lesssim \|\mathcal{A}_r G\|_{L^{p_1}(w^{p_1})}$ for every $(F,G)\in\mathcal{F}$ and for every $w\in A_{p_1,q_1}$. At this point we invoke \cite[Theorem~3.23]{CMP}  ---with the family of pairs $(\mathcal{A}_r F,\mathcal{A}_r G)$, where $(F,G)\in\mathcal{F}$--- to obtain the desired estimate. 
\end{proof}

\begin{proof}[Proof of Corollary~\ref{corol:M-tent}]
	As observed above the case $1<p<\infty$ is particularly easy. Indeed, let $\mathcal{F}$ be the family of pairs $(\mathcal{M}(G),|G|)$, where $G:\ree_+ \to\re$ with $G(\cdot,t)\in L^1_\loc(\re)$. Note that for every $t>0$ and for every $w\in A_2$ by \cite{Buckley}
	\[
	\|\mathcal{M}(G)(\cdot,t)\|_{L^2(w)}
	=
	\|M(G(\cdot,t))\|_{L^2(w)}
	\lesssim
	[w]_{A_2}\,\|G(\cdot,t)\|_{L^2(w)}.
	\]
	We can then invoke Theorem~\ref{theor:extrapol} part (i) with $p_0=2$ to obtain for every $p,r\in (1,\infty)$ and every $w\in A_p$ 
	\[
	\|\mathcal{M}(G)\|_{\mathscr{T}_r^p(w)}
	\lesssim
	\|G\|_{\mathscr{T}_r^p(w)}.
	\]
	
	We are left with considering the case $p=1$. We invoke \cite[Lemma~4.1]{AP} which shows that for every $r\in (1,\infty)$, $t>0$ and $x\in\re^n$ the following estimates holds
	\[
	\Big(\fint_{B(x,t)} Mf(x)^r\,dx\Big)^{\frac1r}
	\lesssim
	\Big(\fint_{B(x,2t)} |f(x)|^r\,dx\Big)^{\frac1r}
	+
	M\Big(\fint_{B(\cdot,t)} |f(z)|\,dz\Big)(x).
	\]
	This implies that
	\[
	\mathcal{A}_r (\mathcal{M}(G))(x)
	\lesssim
	\mathcal{A}_r^2 G(x)
	+
	\Big(\int_0^\infty  \mathcal{M}(\widetilde{G})(x)^r\,\frac{dt}{t}\Big)^{\frac1r}
	\]
	with $\widetilde{G}(x,t)=\fint_{B(x,t)} |G(z,t)|\,dz$, and where 
	\[
	\mathcal{A}_r^2 (G)(x)
	=
	\Big(\iint_{|x-y|<2\,t} |F(y,t)|^r\,\frac{dy\,dt}{t^{n+1}}\Big)^{\frac1r}.
	\]
	Tchebychev's inequality and the change of angle formula (see for instance \cite[Section~3.1]{MP1}) imply 
	\[
	\|\mathcal{A}_r^2 G\|_{L^{1,\infty}(w)}
	\le
	\|\mathcal{A}_r^2 G\|_{L^{1}(w)}
	\lesssim
	\|\mathcal{A}_r G\|_{L^{1}(w)}
	=
	\|G\|_{\mathscr{T}_r^1(w)}.
	\]
	On the other hand the weighted Fefferman-Stein vector-valued weak-type estimates for the Hardy-Littlewood maximal function (see \cite{RRT}) and Jensen's inequality yield
	\begin{multline*}
		\Big\| \Big(\int_0^\infty  \mathcal{M}(\widetilde{G})(\cdot)^r\,\frac{dt}{t}\Big)^{\frac1r}\Big\|_{L^{1,\infty}(w)}
		\lesssim
		\Big\| \Big(\int_0^\infty  \widetilde{G}(\cdot,t)^r\,\frac{dt}{t}\Big)^{\frac1r}\Big\|_{L^{1}(w)}
		\\
		=
		\Big\| \Big(\int_0^\infty \Big( \fint_{B(\cdot,t)} |G(z,t)|\,dz\Big) ^r\,\frac{dt}{t}\Big)^{\frac1r}\Big\|_{L^{1}(w)}
		\le
		\Big\| \Big(\int_0^\infty \fint_{B(\cdot,t)} |G(z,t)|^r\,dz\,\frac{dt}{t}\Big)^{\frac1r}\Big\|_{L^{1}(w)}
		\\
		=
		v_n^{-1} \|\mathcal{A}_r G\|_{L^1(w)}
		=
		v_n^{-1} \|G\|_{\mathscr{T}_r^1(w)},
	\end{multline*}
	where $v_n=|B(0,1)|$. 	All these imply 
	\begin{multline*}
		\|\mathcal{M}(G)\|_{\mathscr{T}_r^{1,\infty}(w)}
		=
		\|\mathcal{A}_r (\mathcal{M}(G))\|_{L^{1,\infty}(w)}
		\lesssim
		\|\mathcal{A}_r^2 (G)\|_{L^{1,\infty}(w)}+
		\Big\| \Big(\int_0^\infty  \mathcal{M}(\widetilde{G})(\cdot)^r\,\frac{dt}{t}\Big)^{\frac1r}\Big\|_{L^{1,\infty}(w)}
		\\
		\lesssim
		\|G\|_{\mathscr{T}_r^1(w)}.
	\end{multline*}
	This completes the proof.
\end{proof}

\section{Off diagonal bounds and Weighted estimates}
\label{section:od}

As we have seen along the paper, Rubio de Francia's extrapolation can be used to obtain estimates in tent spaces once some family of weighted estimates are known. Nonetheless, there are other extrapolation techniques, based on  off-diagonal estimates, which can be used to derive estimates in the scale of tent spaces.

Let us first introduce some notation. Given a family of bounded linear operators $(T_{t})_{t>0}$, we recall that for each $t>0$, we write $\mathcal{T}_t$ for the extension of $T_t$ to functions defined in  $\ree_+$. We then set for each $F:\ree_+\to\re$
\[
\mathbf{T}(F)(x,t)
:=
\mathcal{T}_t(F)(x)
=T_t(F(\cdot,t))(x),
\qquad 
(x,t)\in\ree_+.
\]
We next recall the definition of off-diagonal decay:
\begin{definition}
Let $1<r<\infty$.
A family of bounded linear operators $(T_{t})_{t>0}$ on $L^{r}(\R^{n})$ is said to 
have $L^r-L^r$ off-diagonal decay of order $M>0$ if
$$
\|\mathbf{1}_{E}T_{t}(\mathbf{1}_{F}f)\|_{r} \lesssim \Big(1+\frac{d(E,F)}{t}\Big)^{-M}\|\mathbf{1}_{F}f\|_{r},
$$
for all Borel sets $E,F \subset \R^{n}$, all $t>0$, and all $f \in L^{r}(\R^{n})$.
\end{definition}

In \cite[Theorem~5.2]{HNP} it is shown that  if $(T_{t})_{t>0}$ has $L^2-L^2$ off-diagonal decay of order $M>\frac{n}{\min(p,2)}$, then $\mathbf{T}$ extends to a bounded operator on $\mathscr{T}_2^p$. This result is perhaps the simplest example of an extrapolation result based on off-diagonal bounds. We observe that if $T_{t}\equiv T$ for a given Calder\'on-Zygmund operator $T$, off-diagonal decay is, in general, only of order $\frac{n}{2}$, and \cite[Theorem~5.2]{HNP} does not apply. However, Corollary \ref{corol:main} still gives that $\mathbf{T}$ extends to a bounded operator on $\mathscr{T}_2^p(w)$ for all $w \in A_{p}$.  

In many applications, one deals with families of operators $(T_{t})_{t>0}$ that have  $L^2-L^2$ off-diagonal decay of arbitrary high order (e.g., when the decay is exponential). Our next result shows that, for such families (the precise order needed here being $M>n$), $\mathbf{T}$ extends to a bounded operator on $\mathscr{T}_2^p(w)$ for all $p\in(1,\infty)$ and all $w \in A_{p}$.

\begin{proposition}\label{prop:safrfrw}
	\label{thm:od}
	Let $(T_t)_{t>0}$ be a family of bounded linear operators on $L^{r}(\R^{n})$ with $L^r-L^r$ off-diagonal decay of order $M$ for some $r\ge 1$ and some $M>n/r$.  
	For every $p$ satisfying $n/M< p<\infty$ and every $w\in A_{p\,M/n}$, $\mathbf{T}$ extends to a bounded operator on $\mathscr{T}_r^p(w)$.
\end{proposition}

The role of $L^2-L^2$ off-diagonal and weighted assumptions in extrapolating tent space estimates can thus be summarized as follows: 
\begin{enumerate}
	\item Regardless of the order of off-diagonal decay if weighted bounds are available, then Rubio de Francia's extrapolation gives that $\mathbf{T}$ extends to a bounded operator on $\mathscr{T}_2^p(w)$ for all $p\in(1,\infty)$ and all $w \in A_{p}$ (Corollary \ref{corol:main}). 
	
	\item If the order of off-diagonal decay is greater than $\frac{n}{\min(p,2)}$, then extrapolation based on off-diagonal bounds gives that $\mathbf{T}$ extends to a bounded operator on $\mathscr{T}_2^p$ (\cite[Theorem 5.2]{HNP}). This follows as well from  Proposition \ref{thm:od} by taking $M=\frac{n\,(1+\epsilon)}{\min(p,2)}$ with $0<\epsilon\ll 1$, $r=2$, and $w\equiv 1$.

	\item If the order of off-diagonal decay is greater than or equal to $n$ (that is the case if the decay is exponential), then Proposition \ref{thm:od} with $r=2$ and $M=n$ gives that $\mathbf{T}$ extends to a bounded operator on $\mathscr{T}_2^p(w)$ for all $p\in(1,\infty)$ and all $w \in A_{p}$, without assuming a priori weighted estimates.
\end{enumerate}

\begin{proof}[Proof of Proposition~\ref{prop:safrfrw}]
Fix $n/M< p<\infty$ and $w_0\in A_{p\,M/n}$. Using the openness of the Muckenhoupt classes of weights and that $M>n/r$  we can find $\eta$ with $\frac{n}{\min\{p,r\}\,M}<\eta<1$ so that $w_0\in A_{p\,M\,\eta /n}$. We can then pick $p_0$ such that $\max\{r/p,1\}< p_0< r\,M\,\eta/n$ and $w_0\in A_{p\,p_0/r}$.

Keeping these choices in mind we are going to use Rubio de Francia's extrapolation to obtain our desired estimate. For starters, localizing $f$ on dyadic annuli of the form 
	$C_1(x,t)=B(x,4\,t)$ and $C_j(x,t)=B(x,2^{j+1}\,t)\setminus B(x,2^{j}\,t)$ for $j\ge 2$,
	and using the off-diagonal decay assumption, we first note that
	\[
	\Big(\fint_{B(x,t)} |T_t f(y)|^r\,dy\Big)^{\frac1r}
	\lesssim
	\sum_{j=1}^\infty 2^{-j(M-\frac{n}{r})} \Big(\fint_{B(x,2^j\,t)} |f(y)|^r\,dy\Big)^{\frac1r},
	\qquad
	\text{for every }x\in\re^n,\ t>0.
	\]
Take next an arbitrary $w\in A_{p_0}$ and define $W_t$ as in Lemma \ref{lemma:aver_w}. The fact that $w\in A_{p_0}$ easily implies that 
	\[
	w(\lambda B)
	\le
	[w]_{A_{p_0}}\, \lambda^{n p_{0}} \,w(B),
	\qquad \text{for every ball $B$ and for every }\lambda>1.
	\]
	This readily implies that $W_{\lambda\,t}(x)\le [w]_{A_{p_0}}\,\lambda^{n (p_0-1)} W_t(x)$ for every $x\in\re^n$, $t>0$, and $\lambda>1$. This together with our assumptions give for any fixed $t>0$:
	\begin{align*}
		\|\mathbf{T}(F)(\cdot, t)\|_{L^r(W_t)}
		&=
		\Big(\int_{\re^n} \Big(\fint_{B(x,t)} |T_t(F(\cdot,t))(y)|^r\,dy\Big)\,w(x)\,dx\Big)^\frac1r
		\\
		&\le
		\Big(\int_{\re^n} \Big(\sum_{j=1}^\infty 2^{-j(M-\frac{n}{r})}\,\Big(\fint_{B(x,2^j\,t)} |F(y,t)|^r\,dy\Big)^{\frac1r}\Big)^r\,w(x)\,dx\Big)^\frac1r
		\\
		&\le
		\sum_{j=1}^\infty 2^{-j(M-\frac{n}{r})}\,\Big(\int_{\re^n} \Big(\fint_{B(x,2^j\,t)} |F(y,t)|^r\,dy\Big)\,w(x)\,dx\Big)^\frac1r
		\\
		&=
		\sum_{j=1}^\infty 2^{-j(M-\frac{n}{r})}\, \Big(\int_{\re^n} |F(y,t)|^r\, W_{2^j\,t}(y)\,dy\Big)^\frac1r
		\\
		&\le
		[w]_{A_{p_0}}\,\sum_{j=1}^\infty 2^{-j(M-\frac{p_0\,n}{r})}\, \Big(\int_{\re^n} |F(y,t)|^r\, W_{t}(y)\,dx\Big)^\frac1r
		\\
		&\lesssim
		[w]_{A_{p_0}}\,\|F(\cdot, t)\|_{L^r(W_t)}, 
	\end{align*}
	where we have used the fact that $p_0\,n/r<M\,\eta<M$.
	This and Lemma \ref{lemma:aver_w} allows us to arrive at
	\begin{multline*}
		\int_{\re^n} \mathcal{A}_r(\mathbf{T}(F))(x)^r\,w(x)\,dx
		=
		v_n\int_0^\infty \|\mathbf{T}(F)(\cdot, t)\|_{L^r(W_t)}^r\,\frac{dt}{t}
		\\
		\lesssim
		v_n\,[w]_{A_{p_0}}\, \int_0^\infty \|F(\cdot, t)\|_{L^r(W_t)}^r\,\frac{dt}{t}
		=
		[w]_{A_{p_0}}\,\int_{\re^n} \mathcal{A}_r(F)(x)^r\,w(x)\,dx.
	\end{multline*}
	We can rewrite this as
	\[
	\int_{\re^n} \big(\mathcal{A}_r(\mathbf{T}(F))(x)^{\frac{r}{p_0}}\big)^{p_0}\,w(x)\,dx
	\le
	M_0^{r}\,[w]_{A_{p_0}}\,\int_{\re^n} \big(\mathcal{A}_r(F)(x)^{\frac{r}{p_0}}\big)^{p_0}\,w(x)\,dx.
	\]
	Since $w\in A_{p_0}$ is arbitrary we can invoke Rubio de Francia extrapolation theorem \cite[Theorem~3.9]{CMP}  with the family of pairs $\big(\mathcal{A}_r(\mathbf{T}(F))(x)^{\frac{r}{p_0}},\mathcal{A}_r(F)(x)^{\frac{r}{p_0}}\big)$, and obtain
	\begin{equation}\label{rrrrr}
		\int_{\re^n} \big(\mathcal{A}_r(\mathbf{T}(F))(x)^{\frac{r}{p_0}}\big)^{q}\,v(x)\,dx
		\lesssim 
		\int_{\re^n} \big(\mathcal{A}_r(F)(x)^{\frac{r}{p_0}}\big)^{q}\,v(x)\,dx,
		\end{equation}
		for all $1<q<\infty$ and all $v\in A_q$. 
	
To obtain the desired estimate we select $q:=p\,p_0/r$ and $v=w_0$. Note that our previous choices  guarantee that $1<q<\infty$ and $v\in A_q$. Hence \eqref{rrrrr} holds and the proof is complete. 
\end{proof}
\begin{remark}
	A remaining open question is to determine the optimal level of off-diagonal decay required to extrapolate tent space estimates. Noting that $\mathscr{T}_2^2$ estimates do not require any off-diagonal decay (provided that the family $(T_t)_{t>0}$ is uniformly bounded on $L^2$), one would thus expect the required level of off-diagonal decay for a $\mathscr{T}_p^2$ estimate to improve as $p$ approaches $2$. Unfortunately, our current methods, whether based on the optimal change of angle in tent spaces in \cite{HNP} or on the doubling property of Muckenhoupt weights in Proposition \ref{thm:od}, do not seem to pick up such an  improvement as $p$ approaches $2$. 
\end{remark}

\end{document}